\documentclass[reqno]{amsart}
\usepackage{amssymb,amsmath,amsthm,latexsym,booktabs,todonotes, color, comment, lineno, verbatim}

\theoremstyle{definition}
\newtheorem{definition}{Definition}

\theoremstyle{plain}
\newtheorem{lemma}[definition]{Lemma}

\newtheorem{theorem}[definition]{Theorem}
\newtheorem{corollary}[definition]{Corollary}
\newtheorem{conjecture}[definition]{Conjecture}

\usepackage{todonotes}

\usepackage[margin=1in]{geometry}
\usepackage{amsmath,amssymb,mathtools, comment}
\usepackage{tikz}
\usepackage{xcolor}
\usepackage{booktabs}
\usepackage{enumitem}
\usepackage[hidelinks]{hyperref}

\definecolor{cellblue}{RGB}{220,235,250}
\definecolor{cellgold}{RGB}{252,239,205}
\definecolor{cellgreen}{RGB}{221,242,221}
\definecolor{vertexred}{RGB}{155,30,45}

\newcommand{\eps}{\varepsilon}

\colorlet{cellpeach}{orange!12}
\colorlet{celllavender}{violet!10}
\colorlet{cellrose}{red!9}
\colorlet{cellmint}{green!9}
\colorlet{cellcream}{yellow!12}

\begin{document}


\title{All Polyominoes Are $C_4$-face-magic}
\author{Parikshit Chalise}
\address{Department of Applied Mathematics and Statistics\\
        Johns Hopkins University\\
        Baltimore, MD 21218\\
         USA}
         \email{pchalis1@jhu.edu}
\author{Richard M. Low}
\address{Department of Mathematics\\
         San Jose State University\\
         San Jose, CA 95192\\
         USA}
\email{richard.low@sjsu.edu}
\author{Arman Eisenkolb-Vaithyanathan}
\address{Lynbrook High School, San Jose, CA 95192, USA}
\email{arman.eisenkolbvaithyanathan@gmail.com}


\date{August 9, 2026. Version 1.0}

\begin{abstract}
For a planar graph $G = (V, E)$ embedded in $\mathbb{R}^2$, let $\mathcal{F}(G)$ denote the set of faces of $G$. Then $G$ is called a \textit{$C_n$-face-magic} graph if there exists a bijection $f: V(G) \to \{1, 2, \dots, |V(G)|\}$ such that for any $F \in \mathcal{F}(G)$ with $F \cong C_n$, the sum of all the vertex labels along $C_n$ is a constant $c$. In this paper, we prove that all polyominoes are $C_4$-face-magic.
\end{abstract}
\maketitle

\section{Introduction} \label{Prelim}

In this paper, we consider only finite simple connected graphs unless specified otherwise. Standard graph theory terminology and notation, as found in \cite{Harary94}, are used. 

For a planar graph $G = (V, E)$ embedded in $\mathbb{R}^2$, let $\mathcal{F}(G)$ denote the set of faces of $G$. Then $G$ is called a \textit{$C_n$-face-magic} graph if there exists a bijection $f: V(G) \to \{1, 2, \dots, |V(G)|\}$ such that for any $F \in \mathcal{F}(G)$ with $F \cong C_n$, the sum of all the vertex labels along $C_n$ is a constant $c$.\\

\noindent
\textbf{Definition.}
A \textit{polyomino} is a finite edge-connected union of unit squares in the square lattice $\mathbb{Z}^2$. We identify the polyomino with its associated plane graph, whose vertices are the lattice vertices of the cells and whose edges are the lattice edges. Consequently, every bounded face is bounded by a 4-cycle. When the polyomino has ``holes," other bounded faces may occur.\\

Various researchers \cite{Curran_2, Curran_1, Myers} have studied $C_4$-face-magic labelings of some direct products of graphs on the projective plane, the torus, and the Klein bottle. The general face-magic property was analyzed for various classes of polygonal graphs on the plane in \cite{SLL}. There, the authors made the following conjecture.

\begin{conjecture} \label{conj}
$($\textit{Shiu--Low--Liu} \cite{SLL}$)$.
Every polyomino is $C_4$-face-magic.
\end{conjecture}

\noindent
In Section \ref{Result} of this paper, we prove Conjecture \ref{conj}. In addition, a Java program \cite{Arman_Program} that implements the constructive proof of Conjecture \ref{conj} is provided.

\section{Path decompositions and auxiliary lemmas}
 \label{lemmas}
 
Embed the polyomino $P = (V, E)$ in the first quadrant of the $xy$-plane so that the leftmost edges of $P$ are on the $y$-axis, the bottommost edges of $P$ are on the $x$-axis, and all the vertices of $P$ are at integer lattice points. 

Let $\mathcal{H}$ be the set of components of the spanning subgraph of $P$ formed by the horizontal edges. Each member of $\mathcal{H}$ is a path, and we order its vertices from left to right. The family $\mathcal{V}$ of vertical paths is defined analogously, and we order each vertical path from bottom to top. For each $H \in \mathcal{H}$, define 
\begin{equation}\label{eq:path-sign}
    \eps (H) = (-1)^{a+y},
\end{equation}
where $(a,y)$ is the leftmost vertex of $H$.
 For any $v = (x, y) \in V(P)$, define
 \begin{equation}\label{eq:vertex-sign}
\chi (v) = (-1)^{x+y}.
\end{equation}

For a polyomino $P$, there are two cases to consider: (1) There is a path in $\mathcal{H}$ or $\mathcal{V}$ of odd order; (2) Every path in $\mathcal{H}$ and $\mathcal{V}$ is of even order. Two technical lemmas will be needed in order to prove Conjecture \ref{conj}.  

\begin{lemma}\label{Lemma_AllEven}
Suppose every path in both $\mathcal H$ and $\mathcal V$ has even order. Then
\begin{equation}\label{eq:horizontal-balance}
    \sum_{\substack{H\in\mathcal H\\ \eps(H)=1}} |{V(H)|}
    =
    \sum_{\substack{H\in\mathcal H\\ \eps(H)=-1}} |{V(H)|}
    =
    \frac{|{V(P)}|}2. \nonumber
\end{equation} 
\end{lemma}
\begin{proof}
Form a perfect matching $M_{hor}$ of $V(P)$ by pairing consecutive vertices on every horizontal path in $\mathcal{H}$. In a similar fashion, form another perfect matching $M_{ver}$ of $V(P)$ by pairing consecutive vertices on every vertical  path in $\mathcal{V}$.

Let
\[
    B=\{v\in V(P):\chi(v)=1\},
    \qquad
    W=\{v\in V(P):\chi(v)=-1\}.
\]

\vspace{10pt}

For a matching edge $e$, denote its endpoints in $B$ and $W$ by $b(e)$ and $w(e)$, respectively. Let $x(v)$ be the first coordinate of $v$.

Suppose that $H\in\mathcal H$ has order $2t$. The left endpoint of each of its $t$ edges in $M_{hor}$ has checkerboard sign $\eps(H)$. Hence
\[
    x(w(e))-x(b(e))=\eps(H)
\]
for every $e\in M_{hor}$ contained in $H$. Therefore
\begin{equation}\label{eq:horizontal-displacement}
    \sum_{H\in\mathcal H}\eps(H)|{V(H)|}
    =
    2\sum_{e\in M_{hor}}\bigl(x(w(e))-x(b(e))\bigr).
\end{equation}
Since $M_{hor}$ is a perfect matching,
\begin{equation}\label{eq:coordinate-difference}
    \sum_{e\in M_{hor}}\bigl(x(w(e))-x(b(e))\bigr)
    =
    \sum_{w\in W}x(w)-\sum_{b\in B}x(b).
\end{equation}
The same difference can be evaluated using $M_{ver}$. Every edge of $M_{ver}$ is vertical, so its endpoints have the same first coordinate. Consequently,
\begin{equation}\label{eq:vertical-displacement}
    \sum_{w\in W}x(w)-\sum_{b\in B}x(b)
    =
    \sum_{e\in M_{ver}}\bigl(x(w(e))-x(b(e))\bigr)
    =0.
\end{equation}
It follows from \eqref{eq:horizontal-displacement}--\eqref{eq:vertical-displacement} that
\[
    \sum_{H\in\mathcal H}\eps(H)|{V(H)|}=0,
\]
which establishes the lemma.
\end{proof}

\begin{lemma}\label{AltPath}
Let
\[
    Q_i=q_{i,0}q_{i,1}\cdots q_{i,n_i-1},
    \qquad 0\leq i\leq r-1
\]
be pairwise vertex-disjoint ordered paths $($where paths of order one are allowed$)$. For each $i$, let $\eps_i\in\{1,-1\}$, and $n=\sum_{i=0}^{r-1} n_i$. If there exists an $n_i$ that is odd, or every $n_i$ is even and
\begin{equation}\label{eq:balance-hypothesis}
    \sum_{\eps_i=1}n_i
    =
    \sum_{\eps_i=-1}n_i
    =
    \frac n2, 
\end{equation}

\noindent
then there exists a constant $K \in \mathbb{N}$ and a bijection
\[
    \lambda:\bigcup_{i=0}^{r-1} V(Q_i)\longrightarrow \{1, 2, \dots, n\}
\]
such that

\begin{equation}\label{Eq_4}
\lambda(q_{i,j})+\lambda(q_{i,j+1})
    =
    K-\eps_i(-1)^j
\end{equation}

\vspace{10pt}

\noindent
for every $i \in \{0, 1, \dots, r-1\}$ and every $0\leq j<n_i-1$.  
\end{lemma}

\begin{proof}
\underline{Case (1)}. Suppose that some $n_i$ is odd. Let
\[
    m=\left\lfloor\frac n2\right\rfloor
    \qquad\text{and}\qquad
    K=2m+2.
\]
Define a permutation $a_1\cdots a_n$ of $\{1, 2, \dots, n\}$ by
\[
    a_{2j+1}=2j+1,
    \qquad
 0\leq j<\left\lceil\frac n2\right\rceil
\]
and
\[
    a_{2j+2}=2m-2j,
    \qquad
    0\leq j<m.
\]
Thus
\[
    a_1 a_2 \cdots a_n
    =
    1,2m,3,2m-2,5,2m-4,\ldots,
\]
and
\begin{equation}\label{eq:first-alternating-sequence}
    a_j+a_{j+1}=K-(-1)^{j-1},
    \qquad 1\leq j < n.
\end{equation}

\vspace{10pt}

\noindent
Hence the consecutive sums alternate between $K-1$ and $K+1$.

Choose an odd-order path $Q_s$. Arrange the paths so that all even-order paths with $\eps_i=1$ occur first, followed by $Q_s$, then all even-order paths with $\eps_i=-1$, and finally the remaining odd-order paths. Partition $a_1 \cdots a_n$ into consecutive blocks whose lengths are the orders of the paths in this arrangement.

Consider a block of length $l$. If $l$ is even, then its first and last internal edge-sums in~\eqref{eq:first-alternating-sequence} lie on the same side of $K$, and the first internal edge-sum of the next block lies on that same side. If $l$ is odd, then its first and last internal edge-sums lie on opposite sides of $K$. Reversing such a block therefore allows either value $K-1$ or $K+1$ to be prescribed as its first internal edge-sum; moreover, the next block begins on the opposite side of $K$.

It follows that every even-order block before $Q_s$ begins with sum $K-1$, every even-order block after $Q_s$ begins with sum $K+1$, and each odd-order block can be oriented so that its first internal edge-sum is $K-\eps_i$. Assign each oriented block to the corresponding ordered path. Since the internal edge-sums alternate within every block, \eqref{Eq_4} follows. 

\vspace{10pt}

\underline{Case (2)}. Suppose that every $n_i$ is even and 
\begin{equation}
    \sum_{\eps_i=1}n_i
    =
    \sum_{\eps_i=-1}n_i
    =
    \frac n2. \nonumber
\end{equation}

\vspace{10pt}

\noindent
Since each sum in \eqref{eq:balance-hypothesis} is even, $n=4m$ for some $m$. Set
\[
    K=n+1.
\]
Define for $0\leq j<m$,
\[
    b_{2j+1}=2j+1,
    \qquad
    b_{2j+2}=n-2j-1.
\]
The sequence $b_1,\ldots,b_{n/2}$ consists of all odd integers in $\{1, 2, \dots, n\}$, and
\begin{equation}\label{eq:odd-label-sequence}
    b_j+b_{j+1}=K-(-1)^{j-1},
    \qquad 1\leq j<n/2.
\end{equation}
Set
\[
    c_j=n+1-b_j.
\]
Then $c_1,\ldots,c_{n/2}$ consists of all even integers in $\{1, 2, \dots, n\}$, and
\begin{equation}\label{eq:even-label-sequence}
    c_j+c_{j+1}=K+(-1)^{j-1},
    \qquad 1\leq j<n/2.
\end{equation}

By \eqref{eq:balance-hypothesis}, the paths with $\eps_i=1$ have total order $n/2$. Partition the sequence $b_1,\ldots,b_{n/2}$ into consecutive blocks having their orders. Since every block has even length, each block begins with edge-sum $K-1$. Similarly, partition $c_1,\ldots,c_{n/2}$ into blocks corresponding to the paths with $\eps_i=-1$; each such block begins with edge-sum $K+1$. Assigning the blocks to their corresponding paths gives \eqref{Eq_4}.
\end{proof}


\section{Proof of Conjecture \ref{conj}} \label{Result}

Using Lemmas \ref{Lemma_AllEven} and \ref{AltPath}, we now prove the main result.

\begin{theorem}\label{Theorem_Main}
Every polyomino $P$ is $C_4$-face-magic.
\end{theorem}
\begin{proof}
Let $n=|{V(P)}|$.  For each $H\in\mathcal{H}$, write its vertices from left to right as
\[
(a,y),(a+1,y),\ldots,(a+t-1,y).
\]
Assign to each $H\in\mathcal H$ the sign $\eps(H)$ as defined in \eqref{eq:path-sign}.

\vspace{10pt} 

\underline{Case (1)}. Suppose first that some horizontal path has odd order. By Lemma \ref{AltPath}, we obtain an integer $K$ and a bijection
\[
    \lambda:V(P)\longrightarrow \{1, 2, \dots, n\}
\]
such that for every $H\in\mathcal H$, we have for all $0\leq j<t-1$,
\begin{equation}\label{eq:horizontal-edge-sum}
    \lambda(a+j,y)+\lambda(a+j+1,y)
    =K-\eps(H)(-1)^j\\
    =K-(-1)^{a+j+y}.
\end{equation}
If $C$ is a cell with lower-left vertex $(x,y)$, then \eqref{eq:horizontal-edge-sum} implies
\begin{equation}\label{eq:cell-sum}
    \sum_{v\in V(C)}\lambda(v)=    K-(-1)^{x+y} + K-(-1)^{x+y+1} =2K.
\end{equation}
In other words, the sum of four vertex labels of each cell $C$ is equal to $2K$.

\vspace{10pt} 

\underline{Case (2)}.
If every horizontal path has even order but some vertical path has odd order, apply the argument in Case (1) to the $90^\circ$ rotation of $P$.

\vspace{10pt} 

\underline{Case (3)}.
Now suppose every $H\in \mathcal{H}$ and every $V\in \mathcal{V}$ is of even order. By Lemma~\ref{Lemma_AllEven},
\[
    \sum_{\substack{H\in\mathcal H\\ \eps(H)=1}}|{V(H)}|
    =
    \sum_{\substack{H\in\mathcal H\\ \eps(H)=-1}}|{V(H)}|
    =
    \frac n2.
\]
Applying Lemma~\ref{AltPath} yields \eqref{eq:horizontal-edge-sum}, which in turn implies \eqref{eq:cell-sum}.

In every case, $\lambda$ is a bijection from $V(P)$ to $\{1,2,\dots,n\}$ for which the sum of the four vertex labels of every cell is constant. Therefore, $P$ is $C_4$-face-magic.
\end{proof}

\noindent
\underline{Remarks}. The reader should note that the proofs of Lemmas \ref{Lemma_AllEven} and \ref{AltPath} and of Theorem \ref{Theorem_Main} do not assume that a polyomino is simply connected. Example 3 in Section \ref{Examples} illustrates a $C_4$-face-magic labeling of a polyomino that is not simply connected. Moreover, all of the arguments and constructions in the proofs of Lemmas \ref{Lemma_AllEven} and \ref{AltPath} and of Theorem \ref{Theorem_Main} hold for general polyominoes in which the cells can be vertex-connected as well as edge-connected.\\

\begin{corollary}\label{All_Polyom}
If every component of a graph $G$ is a polyomino, then $G$ is $C_4$-face-magic. 
\end{corollary}
\begin{proof}
Form a single polyomino (with the same vertex set as $G$) by adding vertical edges and/or horizontal edges between the successive components of $G$. See Figure \ref{fig:complex-identifications}. Then the application of Lemmas \ref{Lemma_AllEven}, \ref{AltPath}, and Theorem \ref{Theorem_Main} provides a $C_4$-face-magic labeling of this polyomino. Finally, remove the added vertical and/or horizontal edges. This gives a $C_4$-face-magic labeling of $G$.
\end{proof}

\vspace{10pt}


\begin{figure}[htbp]
\centering

\begin{tikzpicture}[
    x=0.75cm,
    y=0.75cm,
    line cap=round,
    line join=round,
    edge/.style={
        draw=blue!70!black,
        line width=0.9pt
    },
    identification/.style={
        edge,
        dashed,
        dash pattern=on 3pt off 2.5pt
    },
    vertex/.style={
        circle,
        fill=blue!70!black,
        draw=blue!70!black,
        inner sep=1.3pt
    }
]

\begin{scope}[shift={(0,0)}]

\draw[edge] (0,0) -- (2,0);
\draw[edge] (0,1) -- (2,1);
\draw[edge] (0,2) -- (1,2);

\draw[edge] (0,0) -- (0,2);
\draw[edge] (1,0) -- (1,2);
\draw[edge] (2,0) -- (2,1);

\foreach \x/\y in {
    0/0,1/0,2/0,
    0/1,1/1,2/1,
    0/2,1/2
}{
    \node[vertex] at (\x,\y) {};
}

\end{scope}

\begin{scope}[shift={(3.3,0)}]

\draw[edge] (0,0) -- (3,0);
\draw[edge] (0,1) -- (3,1);
\draw[edge] (0,2) -- (2,2);

\draw[edge] (0,0) -- (0,2);
\draw[edge] (1,0) -- (1,2);
\draw[edge] (2,0) -- (2,2);
\draw[edge] (3,0) -- (3,1);

\foreach \x/\y in {
    0/0,1/0,2/0,3/0,
    0/1,1/1,2/1,3/1,
    0/2,1/2,2/2
}{
    \node[vertex] at (\x,\y) {};
}

\end{scope}

\draw[edge,->,>=stealth]
    (6.9,1) -- (8.2,1);

\begin{scope}[shift={(8.8,0)}]

\draw[edge] (0,0) -- (2,0);
\draw[edge] (0,1) -- (2,1);
\draw[edge] (0,2) -- (1,2);

\draw[edge] (0,0) -- (0,2);
\draw[edge] (1,0) -- (1,2);
\draw[edge] (2,0) -- (2,1);

\foreach \x/\y in {
    0/0,1/0,2/0,
    0/1,1/1,2/1,
    0/2,1/2
}{
    \node[vertex] at (\x,\y) {};
}

\end{scope}

\begin{scope}[shift={(11.8,0)}]

\draw[edge] (0,0) -- (3,0);
\draw[edge] (0,1) -- (3,1);
\draw[edge] (0,2) -- (2,2);

\draw[edge] (0,0) -- (0,2);
\draw[edge] (1,0) -- (1,2);
\draw[edge] (2,0) -- (2,2);
\draw[edge] (3,0) -- (3,1);

\foreach \x/\y in {
    0/0,1/0,2/0,3/0,
    0/1,1/1,2/1,3/1,
    0/2,1/2,2/2
}{
    \node[vertex] at (\x,\y) {};
}

\end{scope}

\draw[identification] (10.8,0) -- (11.8,0);
\draw[identification] (10.8,1) -- (11.8,1);

\end{tikzpicture}

\caption{Forming a single polyomino by adding two horizontal edges between two disjoint polyominoes}
\label{fig:complex-identifications}
\end{figure}
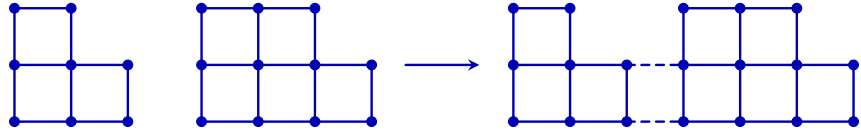


\section{Illustrative examples}\label{Examples}


\underline{\textbf{Example 1}}. Consider the five-cell polyomino
\[
P=\{(0,0),(1,0),(2,0),(0,1),(1,1)\},
\]
where each ordered pair denotes the lower-left corner of a cell. See Figure \ref{fig:five-cell-complex}.

\vspace{10pt}


\begin{figure}[htbp]
\centering

\begin{tikzpicture}[scale=1.5]

\fill[blue!12] (0,0) rectangle (1,1);
\fill[blue!12] (1,0) rectangle (2,1);
\fill[blue!12] (2,0) rectangle (3,1);
\fill[blue!12] (0,1) rectangle (1,2);
\fill[blue!12] (1,1) rectangle (2,2);

\draw[thick] (0,0)--(3,0);
\draw[thick] (0,1)--(3,1);
\draw[thick] (0,2)--(2,2);

\draw[thick] (0,0)--(0,2);
\draw[thick] (1,0)--(1,2);
\draw[thick] (2,0)--(2,2);
\draw[thick] (3,0)--(3,1);

\fill (0,0) circle (1.8pt);
\fill (1,0) circle (1.8pt);
\fill (2,0) circle (1.8pt);
\fill (3,0) circle (1.8pt);

\fill (0,1) circle (1.8pt);
\fill (1,1) circle (1.8pt);
\fill (2,1) circle (1.8pt);
\fill (3,1) circle (1.8pt);

\fill (0,2) circle (1.8pt);
\fill (1,2) circle (1.8pt);
\fill (2,2) circle (1.8pt);

\node at (0.5,0.5) {$C_1$};
\node at (1.5,0.5) {$C_2$};
\node at (2.5,0.5) {$C_3$};
\node at (0.5,1.5) {$C_4$};
\node at (1.5,1.5) {$C_5$};

\end{tikzpicture}

\caption{The five-cell polyomino $P$}
\label{fig:five-cell-complex}
\end{figure}
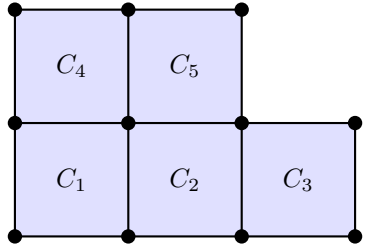


\noindent
This polyomino has \(n=11\) vertices. Its horizontal paths \(H_0,H_1,H_2\)
have orders \(4,4,3\), respectively, and
\[
\eps(H_0)=1,\qquad
\eps(H_1)=-1,\qquad
\eps(H_2)=1.
\]

Since \(H_2\) has odd order, Case~(1) of Lemma \ref{AltPath} applies. Here
\[
K=12,
\]
and the alternating permutation is
\[
1,10,3,8,5,6,7,4,9,2,11.
\]
Its consecutive sums alternate between
\[
K-1=11
\qquad\text{and}\qquad
K+1=13.
\]

Ordering the paths as \(H_0,H_2,H_1\) and partitioning the permutation
into blocks of lengths \(4,3,4\) gives
\[
H_0:(1,10,3,8),\qquad
H_2:(5,6,7),\qquad
H_1:(4,9,2,11).
\]
Thus the consecutive edge sums along each horizontal path alternate
between \(11\) and \(13\), with the initial phase prescribed by
\(\eps(H)\), as required by Lemma \ref{AltPath}.

For every cell, its lower and upper horizontal edges therefore have
complementary sums \(11\) and \(13\). Hence every cell has vertex-label
sum
\[
11+13=24=2K.
\]
Figure \ref{fig:labeled-five-cell-polyomino} illustrates the use of the odd-order branch of Lemma \ref{AltPath} in the proof of Theorem \ref{Theorem_Main}.

\vspace{10pt}


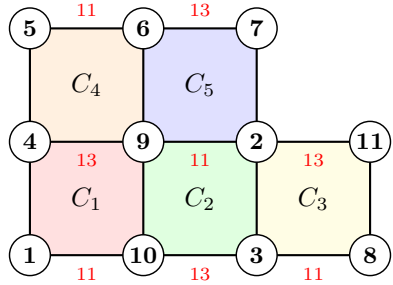
\begin{figure}[htbp]
\centering

\begin{tikzpicture}[
    scale=1.5,
    vertex/.style={
        circle,
        draw=black,
        fill=white,
        line width=0.55pt,
        minimum size=5.4mm,
        inner sep=0pt,
        font=\small
    },
    edgesum/.style={
        text=red,
        font=\scriptsize\bfseries,
        inner sep=0pt
    }
]

\fill[red!12]    (0,0) rectangle (1,1); 
\fill[green!12]  (1,0) rectangle (2,1); 
\fill[yellow!12] (2,0) rectangle (3,1); 
\fill[orange!12] (0,1) rectangle (1,2); 
\fill[blue!12]   (1,1) rectangle (2,2); 

\draw[thick] (0,0)--(3,0);
\draw[thick] (0,1)--(3,1);
\draw[thick] (0,2)--(2,2);

\draw[thick] (0,0)--(0,2);
\draw[thick] (1,0)--(1,2);
\draw[thick] (2,0)--(2,2);
\draw[thick] (3,0)--(3,1);

\node at (0.5,0.5) {$C_1$};
\node at (1.5,0.5) {$C_2$};
\node at (2.5,0.5) {$C_3$};
\node at (0.5,1.5) {$C_4$};
\node at (1.5,1.5) {$C_5$};

\node[vertex] at (0,0) {$\mathbf{1}$};
\node[vertex] at (1,0) {$\mathbf{10}$};
\node[vertex] at (2,0) {$\mathbf{3}$};
\node[vertex] at (3,0) {$\mathbf{8}$};

\node[vertex] at (0,1) {$\mathbf{4}$};
\node[vertex] at (1,1) {$\mathbf{9}$};
\node[vertex] at (2,1) {$\mathbf{2}$};
\node[vertex] at (3,1) {$\mathbf{11}$};

\node[vertex] at (0,2) {$\mathbf{5}$};
\node[vertex] at (1,2) {$\mathbf{6}$};
\node[vertex] at (2,2) {$\mathbf{7}$};

\node[edgesum, yshift=7pt] at (0.5,2) {$11$};
\node[edgesum, yshift=7pt] at (1.5,2) {$13$};

\node[edgesum, yshift=-7pt] at (0.5,1) {$13$};
\node[edgesum, yshift=-7pt] at (1.5,1) {$11$};
\node[edgesum, yshift=-7pt] at (2.5,1) {$13$};

\node[edgesum, yshift=-7pt] at (0.5,0) {$11$};
\node[edgesum, yshift=-7pt] at (1.5,0) {$13$};
\node[edgesum, yshift=-7pt] at (2.5,0) {$11$};

\end{tikzpicture}

\caption{A $C_4$-face-magic labeling of $P$}
\label{fig:labeled-five-cell-polyomino}

\end{figure}


\hfill $\lozenge$

\vspace{10pt}

\underline{\textbf{Example 2}}. Consider the plus-pentomino, which has \(n=12\) vertices. 

\vspace{10pt}


\begin{figure}[htbp]
\begin{center}

\begin{tikzpicture}[scale=1.5]

\fill[blue!12] (1,2) rectangle (2,3); 
\fill[blue!12] (0,1) rectangle (1,2); 
\fill[blue!12] (1,1) rectangle (2,2); 
\fill[blue!12] (2,1) rectangle (3,2); 
\fill[blue!12] (1,0) rectangle (2,1); 

\draw[thick] (1,3)--(2,3);
\draw[thick] (0,2)--(3,2);
\draw[thick] (0,1)--(3,1);
\draw[thick] (1,0)--(2,0);

\draw[thick] (0,1)--(0,2);
\draw[thick] (1,0)--(1,3);
\draw[thick] (2,0)--(2,3);
\draw[thick] (3,1)--(3,2);

\foreach \x/\y in {
    1/3, 2/3,
    0/2, 1/2, 2/2, 3/2,
    0/1, 1/1, 2/1, 3/1,
    1/0, 2/0
}{
    \fill (\x,\y) circle (1.8pt);
}

\node at (1.5,2.5) {$C_1$};
\node at (0.5,1.5) {$C_2$};
\node at (1.5,1.5) {$C_3$};
\node at (2.5,1.5) {$C_4$};
\node at (1.5,0.5) {$C_5$};

\end{tikzpicture}

\end{center}
\caption{The plus-pentomino}
\label{fig:cell-complex}
\end{figure}
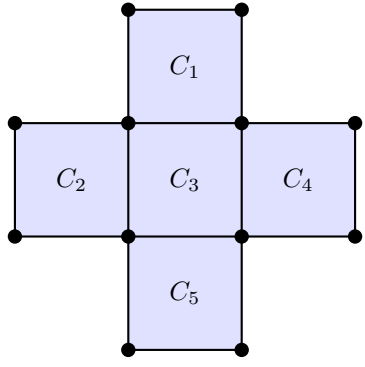


\noindent
Its horizontal
paths have orders
\[
2,4,4,2
\]
and signs
\[
-1,-1,+1,+1.
\]
Every horizontal and vertical path has even order. Moreover,
\[
\sum_{\eps(H)=1}|V(H)|
=
\sum_{\eps(H)=-1}|V(H)|
=
6
=
\frac{n}{2},
\]
in accordance with Lemma \ref{Lemma_AllEven}. Hence the all-even case of Lemma~\ref{AltPath} applies.

In this case,
\[
K=n+1=13.
\]
The odd labels are arranged in the sequence 
\[
1,11,3,9,5,7,
\]
whose consecutive sums alternate between \(12\) and \(14\), while the
even labels are arranged in the sequence 
\[
12,2,10,4,8,6,
\]
whose consecutive sums alternate between \(14\) and \(12\).

Partitioning the odd sequence into blocks of lengths \(4,2\) and the
even sequence into blocks of lengths \(2,4\) yields
\[
\begin{aligned}
H_0&:(12,2),&
H_1&:(10,4,8,6),\\
H_2&:(1,11,3,9),&
H_3&:(5,7).
\end{aligned}
\]
The resulting vertex labeling may be displayed as
\[
\begin{array}{ccccc}
 & 5 & 7 &  &  \\[2pt]
1 & 11 & 3 & 9 \\[2pt]
10 & 4 & 8 & 6 \\[2pt]
 & 12 & 2 &
\end{array}
\]
with the entries placed at their corresponding lattice vertices.

Each negative-sign horizontal path begins with edge sum
\[
K+1=14,
\]
and each positive-sign horizontal path begins with edge sum
\[
K-1=12.
\]
Consequently, the lower and upper horizontal edges of every cell have
complementary sums \(12\) and \(14\), so every cell has vertex-label sum
\[
12+14=26=2K.
\]

This illustrates how Lemma \ref{Lemma_AllEven} supplies the balance condition needed for
the all-even case of Lemma \ref{AltPath}, which, in turn, produces the
\(C_4\)-face-magic labeling used in Theorem \ref{Theorem_Main}.



\begin{figure}[htbp]
\centering

\begin{tikzpicture}[
    scale=1.5,
    vertex/.style={
        circle,
        draw=black,
        fill=white,
        line width=0.55pt,
        minimum size=5.4mm,
        inner sep=0pt,
        font=\small
    },
    edgesum/.style={
        text=red,
        font=\scriptsize\bfseries,
        inner sep=0pt
    }
]

\fill[red!12]    (1,2) rectangle (2,3); 
\fill[green!12]  (0,1) rectangle (1,2); 
\fill[yellow!12] (1,1) rectangle (2,2); 
\fill[orange!12] (2,1) rectangle (3,2); 
\fill[blue!12]   (1,0) rectangle (2,1); 

\draw[thick] (1,3)--(2,3);

\draw[thick] (0,2)--(3,2);
\draw[thick] (0,1)--(3,1);

\draw[thick] (1,0)--(2,0);

\draw[thick] (0,1)--(0,2);
\draw[thick] (1,0)--(1,3);
\draw[thick] (2,0)--(2,3);
\draw[thick] (3,1)--(3,2);

\node at (1.5,2.5) {$C_1$};
\node at (0.5,1.5) {$C_2$};
\node at (1.5,1.5) {$C_3$};
\node at (2.5,1.5) {$C_4$};
\node at (1.5,0.5) {$C_5$};

\node[vertex] at (1,3) {$\mathbf{5}$};
\node[vertex] at (2,3) {$\mathbf{7}$};

\node[vertex] at (0,2) {$\mathbf{1}$};
\node[vertex] at (1,2) {$\mathbf{11}$};
\node[vertex] at (2,2) {$\mathbf{3}$};
\node[vertex] at (3,2) {$\mathbf{9}$};

\node[vertex] at (0,1) {$\mathbf{10}$};
\node[vertex] at (1,1) {$\mathbf{4}$};
\node[vertex] at (2,1) {$\mathbf{8}$};
\node[vertex] at (3,1) {$\mathbf{6}$};

\node[vertex] at (1,0) {$\mathbf{12}$};
\node[vertex] at (2,0) {$\mathbf{2}$};

\node[edgesum, yshift=7pt]  at (1.5,3) {$12$};

\node[edgesum, yshift=7pt]  at (0.5,2) {$12$};
\node[edgesum, yshift=7pt]  at (1.5,2) {$14$};
\node[edgesum, yshift=7pt]  at (2.5,2) {$12$};

\node[edgesum, yshift=-7pt] at (0.5,1) {$14$};
\node[edgesum, yshift=-7pt] at (1.5,1) {$12$};
\node[edgesum, yshift=-7pt] at (2.5,1) {$14$};

\node[edgesum, yshift=-7pt] at (1.5,0) {$14$};

\end{tikzpicture}

\caption{A $C_4$-face-magic labeling of the plus-pentomino, obtained from the all-even case of Lemma~\ref{AltPath}}
\label{fig:labeled-cell-complex}
\end{figure}
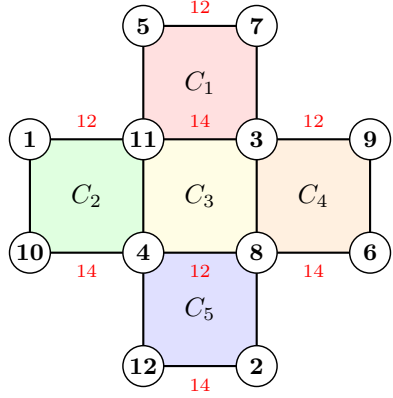


\vspace{10pt}

\hfill $\lozenge$

\vspace{10pt}

\indent
\underline{\textbf{Example 3}}. Figure \ref{fig:graph-face-values} illustrates a remark made earlier in the paper.

\vspace{10pt}


\begin{figure}[htbp]
\centering

\begin{tikzpicture}[
  y=1.5cm,
  x=1.5cm,
  edge/.style={draw=black, line width=0.55pt},
  vertex/.style={
    circle,
    draw=black,
    fill=white,
    line width=0.55pt,
    minimum size=5.4mm,
    inner sep=0pt,
    font=\small
  },
  facevalue/.style={font=\fontsize{7}{8}\selectfont, text=red}
]

\draw[edge] (0,4)--(1,4)--(2,4)--(3,4)--(4,4);
\draw[edge] (0,3)--(1,3)--(2,3)--(3,3)--(4,3);

\draw[edge] (0,2)--(1,2);
\draw[edge] (3,2)--(4,2);
\draw[edge] (5,2)--(6,2);

\draw[edge] (0,1)--(1,1)--(2,1)--(3,1)--(4,1)--(5,1)--(6,1);
\draw[edge] (0,0)--(1,0)--(2,0)--(3,0)--(4,0)--(5,0)--(6,0);

\foreach \x in {0,1,2,3,4} {
  \draw[edge] (\x,4)--(\x,3);
}

\foreach \x in {0,1,3,4} {
  \draw[edge] (\x,3)--(\x,2);
}

\foreach \x in {0,1,3,4,5,6} {
  \draw[edge] (\x,2)--(\x,1);
}

\foreach \x in {0,1,2,3,4,5,6} {
  \draw[edge] (\x,1)--(\x,0);
}

\foreach \x in {0.5,1.5,2.5,3.5} {
  \node[facevalue] at (\x,3.5) {64};
}
\node[facevalue] at (0.5,2.5) {64};
\node[facevalue] at (3.5,2.5) {64};
\node[facevalue] at (0.5,1.5) {64};
\node[facevalue] at (3.5,1.5) {64};
\node[facevalue] at (5.5,1.5) {64};
\foreach \x in {0.5,1.5,2.5,3.5,4.5,5.5} {
  \node[facevalue] at (\x,0.5) {64};
}

\node[vertex] at (0,4) {$\mathbf{2}$};
\node[vertex] at (1,4) {$\mathbf{29}$};
\node[vertex] at (2,4) {$\mathbf{4}$};
\node[vertex] at (3,4) {$\mathbf{27}$};
\node[vertex] at (4,4) {$\mathbf{6}$};

\node[vertex] at (0,3) {$\mathbf{25}$};
\node[vertex] at (1,3) {$\mathbf{8}$};
\node[vertex] at (2,3) {$\mathbf{23}$};
\node[vertex] at (3,3) {$\mathbf{10}$};
\node[vertex] at (4,3) {$\mathbf{21}$};

\node[vertex] at (0,2) {$\mathbf{1}$};
\node[vertex] at (1,2) {$\mathbf{30}$};
\node[vertex] at (3,2) {$\mathbf{22}$};
\node[vertex] at (4,2) {$\mathbf{11}$};
\node[vertex] at (5,2) {$\mathbf{20}$};
\node[vertex] at (6,2) {$\mathbf{13}$};

\node[vertex] at (0,1) {$\mathbf{18}$};
\node[vertex] at (1,1) {$\mathbf{15}$};
\node[vertex] at (2,1) {$\mathbf{16}$};
\node[vertex] at (3,1) {$\mathbf{17}$};
\node[vertex] at (4,1) {$\mathbf{14}$};
\node[vertex] at (5,1) {$\mathbf{19}$};
\node[vertex] at (6,1) {$\mathbf{12}$};

\node[vertex] at (0,0) {$\mathbf{3}$};
\node[vertex] at (1,0) {$\mathbf{28}$};
\node[vertex] at (2,0) {$\mathbf{5}$};
\node[vertex] at (3,0) {$\mathbf{26}$};
\node[vertex] at (4,0) {$\mathbf{7}$};
\node[vertex] at (5,0) {$\mathbf{24}$};
\node[vertex] at (6,0) {$\mathbf{9}$};

\end{tikzpicture}

\caption{A $C_4$-face-magic labeling of a polyomino with a ``hole" in it}
\label{fig:graph-face-values}
\end{figure}
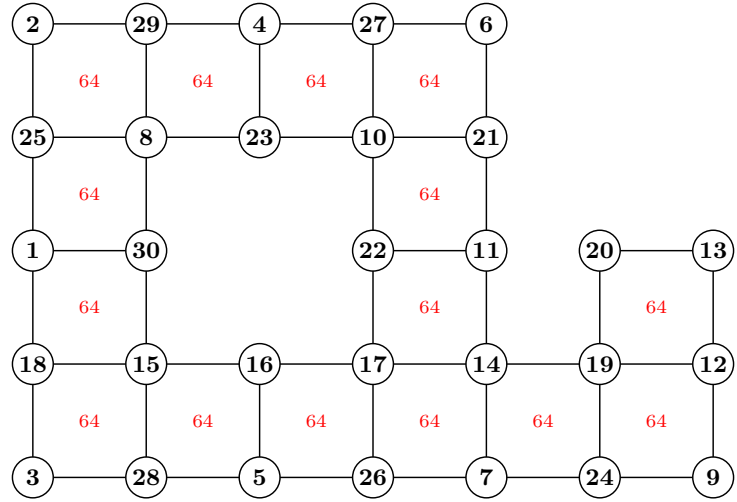

\hfill $\lozenge$


\section{Future directions}\label{OpenQuestions}
Let $\mathcal{C}(P)$ denote the set of cells of a polyomino $P$, and
let $n=|V(P)|$. Write $[n]=\{1,2,\ldots,n\}$. We define the \emph{$C_4$-face-magic spectrum} of $P$
by
\[
\operatorname{Spec}_{C_4}(P)
=
\left\{
s\in\mathbb N:
\exists\text{ a bijection }\lambda:V(P)\to[n]
\text{ such that }
\sum_{v\in V(C)}\lambda(v)=s
\text{ for every }C\in\mathcal C(P)
\right\}.
\]
Thus, $\operatorname{Spec}_{C_4}(P)$ is the set of all possible magic
constants of $C_4$-face-magic labelings of $P$. By
Theorem~\ref{Theorem_Main}, this set is nonempty for every polyomino.

The spectrum has a natural symmetry. Indeed, if $\lambda$ is a
$C_4$-face-magic labeling with magic constant $s$, then the
complementary labeling $\overline{\lambda}(v)=n+1-\lambda(v)$
has magic constant $4(n+1)-s$.
Consequently, $s\in\operatorname{Spec}_{C_4}(P)
$ if and only if $4(n+1)-s\in\operatorname{Spec}_{C_4}(P)$.
Moreover,
\[
\operatorname{Spec}_{C_4}(P)
\subseteq
\{10,11,\ldots,4n-6\},
\]
since the four vertices of each cell receive distinct labels. We conclude with the following questions.

\begin{enumerate}
    \item Determine $\operatorname{Spec}_{C_4}(P)$ for some class of
    polyominoes $P$.

    \item Is $\operatorname{Spec}_{C_4}(P)$ always an interval of
    integers?

    \item Characterize the polyominoes $P$ for which $        2(n+1)\in\operatorname{Spec}_{C_4}(P).$ 

    \item Determine the minimum and maximum elements of
    $\operatorname{Spec}_{C_4}(P)$, perhaps in terms of structural properties
    of $P$.

    \item Explore similar questions for $C_n$-face-magic graphs for $n \geq 5$.
\end{enumerate}

\section{Tool and computational resource disclosure}\label{Disclosure}

An initial proof of Conjecture~\ref{conj}, including versions
of Lemmas~\ref{Lemma_AllEven} and~\ref{AltPath}, was generated by
OpenAI's GPT-5.6 Sol. For
this proof search, we adapted the prompt that led to a purported proof of the longstanding \emph{Cycle Double
Cover Conjecture}. The full original prompt for the CDCC is available at \cite{OpenAI-CDCPrompt}.


The authors independently verified all mathematical arguments and rewrote the AI-assisted material; they take full responsibility for the final contents of the paper.


\begin{thebibliography}{99}

\bibitem{Curran_2}
Curran, S.J., Odd order $C_4$-face-magic projective grid graphs. \textit{Electron. J. Graph Theory Appl.} 14 (2026), no. 1, 153-178.

\bibitem{Curran_1}
Curran, S.J., Low, R.M., and Locke, S.C., $C_4$-face-magic toroidal labelings on $C_{m} \times C_{n}$. \textit{Art Discrete Appl. Math}. 4 (2021), no. 1, \#P1.04, 33pp. 

\bibitem{Arman_Program}
Eisenkolb-Vaithyanathan, A.,
\emph{$C_4$-Face-Magic Polyomino Labeler}, GitHub repository, 2026.
\url{https://github.com/ArmanE-V/C4-Face-Magic-Labeler} (accessed August 7, 2026).

\bibitem{Harary94}
Harary, F., \textit{Graph Theory}. Addison-Wesley, Reading, MA (1994).

\bibitem{Myers}
Myers, T. and Curran, S.J., $C_4$-face-magic labeling on a $4 \times 4$ Klein bottle grid graph. arXiv:2606.06817v1.

\bibitem{OpenAI-CDCPrompt}
OpenAI,
\emph{Prompt Used for ``A Proof of the Cycle Double Cover Conjecture,''}
2026.
Available at:
\url{https://cdn.openai.com/pdf/04d1d1e4-bc75-476a-97cf-49055cd98d31/cdc_prompt.pdf}
(accessed August 6, 2026).

\bibitem{SLL}
Shiu, W.C., Low, R.M., and Liu, A.K., Face-magic labelings of polygonal graphs. \textit{Theory Appl. Graphs}. 11 (2024), no. 1, Art. 7, 14pp.

    
\end{thebibliography}
\end{document}